\documentclass[12pt]{article}
\usepackage{breakurl}
\usepackage{amsmath,amssymb,amsthm,bbm,natbib,color,hyperref}
\usepackage[margin=1.1in]{geometry}
\usepackage{caption}
\usepackage{graphicx}
\usepackage{subcaption}
\usepackage{afterpage}
\usepackage[hang,flushmargin]{footmisc}

\newtheorem{theorem}{Theorem}

\newtheorem{proposition}[theorem]{Proposition}

\newtheorem{example}[theorem]{Example}

\newcommand{\R}{\mathbb{R}}
\newcommand{\eps}{\epsilon}

\newcommand{\EE}[1]{\mathbb{E}\left[{#1}\right]}
\newcommand{\EEst}[2]{\mathbb{E}\left[{#1}\  \middle| \ {#2}\right]}

\newcommand{\PP}[1]{\mathbb{P}\left\{{#1}\right\}}
\newcommand{\PPst}[2]{\mathbb{P}\left\{{#1}\  \middle| \ {#2}\right\}}

\newcommand{\eqd}{\stackrel{\textnormal{d}}{=}}

\newcommand{\iidsim}{\stackrel{\textnormal{iid}}{\sim}}

\newcommand{\Fcal}{\mathcal{F}}
\title{Hoeffding and Bernstein inequalities for weighted sums of exchangeable random variables}
\author{Rina Foygel Barber\thanks{Department of Statistics, University of Chicago}}
\date{\today}
\begin{document}
\maketitle

\begin{abstract}
The aim of this paper is to establish Hoeffding and Bernstein type concentration inequalities
for weighted sums of exchangeable random variables. A special case is the i.i.d.\ setting,
where random variables are sampled independently from some distribution (and are therefore
exchangeable). In contrast to 
the existing literature on this problem, our results provide a natural unified view of both 
the i.i.d.\ and the exchangeable setting.
\end{abstract}

\section{Introduction}\label{sec:intro}

Let $w_1,\hdots,w_n\in\R$ be fixed weights and let $X_1,\hdots,X_n,X_{n+1},\hdots,X_N\in[-1,1]$ be exchangeable random variables.
Our goal is to establish concentration properties for the weighted sum
\[w_1 X_1 + \hdots + w_n X_n.\]
In the special case where the $X_i$'s are i.i.d.\ (a strictly stronger condition),
there are a wide range of well-known concentration bounds for this sum---for example, the Hoeffding and Bernstein inequalities,
as well as multiple more technical results providing tighter bounds.
There are also many recent results handling the exchangeable case---the $X_i$'s are no longer assumed to be independent
but are required to satisfy exchangeability, 
meaning that the joint distribution of $(X_1,\hdots,X_N)$ is equal to that of any permutation, $(X_{\pi(1)},\hdots,X_{\pi(N)})$.
In particular, tight bounds have been established for the special case 
where $w_1 = \hdots = w_n$ (for example,
if $w_i \equiv 1/n$, then the weighted sum can be viewed as a sample mean, for a sample of size $n$ drawn without replacement
from the finite population $X_1,\hdots,X_N$).

However, for the general case where the $w_i$'s may be nonconstant,
a key gap remains in the current literature on this problem, as we now describe.
Since the i.i.d.\ assumption is a special case of exchangeability,
we might expect that classical results for the i.i.d.\ setting (i.e., for the sum $w_1X_1 + \hdots + w_nX_n$,
where $X_i\iidsim P$ for some distribution $P$) can be recovered as a special case
of the concentration bounds for exchangeable sequences.
Indeed, if we take $N\to\infty$, then an infinite exchangeable sequence $X_1,X_2,\hdots$
must be representable as an i.i.d.\ draw from some (random) distribution $P$, by De Finetti's theorem 
\citep{de1929funzione,hewitt1955symmetric}. Even if $N$ is finite, 
there are finite versions of De Finetti's theorem establishing that if $N\gg n$ then the joint distribution 
of an exchangeable sequence $X_1,\hdots,X_n$ can be well approximated as
an i.i.d.\ draw from some (random) $P$---the classical result of this type was established 
by \citet{diaconis1980finite}, proving a bound in the total variation distance, while a more modern information-theoretic version of the finite De Finetti 
 theorem can be found in \citet{gavalakis2021information}, bounding the Kullback--Leibler (KL) divergence.
Consequently, concentration results on the weighted sum $\sum_{i=1}^n w_i X_i$,
under an assumption of exchangeability for $X_1,\hdots,X_N$, should ideally recover
the classical bounds for the i.i.d.\ case when we take $N\to\infty$. But, as we will see below,
this is not the case for existing bounds when the weights $w_1,\hdots,w_n$ are arbitrary.

The primary goal of this paper is to close this gap, by deriving concentration bounds that 
achieve the following aims.
\begin{quote}
\textbf{Main aims: }We seek to establish concentration results for the weighted sum $w_1X_1 + \hdots + w_nX_n$
 under the assumption of exchangeability, such that
for a fixed number of terms $n$ in the sum, the results
\begin{itemize}
\item[(a)] Provide (approximately) tight bounds for fixed $N$; and
\item[(b)] Recover classical bounds for the i.i.d.\ setting, in the limit where $N\to \infty$.
\end{itemize}
Moreover, since the distribution of the weighted sum $\sum_{i=1}^n w_i X_i$
remains the same if we permute the weights,  for interpretability our results should
\begin{itemize}
\item[(c)] Derive bounds that are symmetric in the weights $w_1,\hdots,w_n$.
\end{itemize}
\end{quote}
The results of this paper will establish both Hoeffding and Bernstein type bounds for the weighted sum
in the exchangeable setting, which achieve these aims.
In addition to their theoretical interest, these results
 have the potential to be applicable across a broad range of problems in statistics, 
 since finite exchangeability arises in many contexts in modern statistical analysis, including
 settings such as
permutation testing \citep{hemerik2018exact,ramdas2023permutation}, conformal prediction \citep{vovk2005algorithmic,lei2018distribution,angelopoulos2023conformal},
and inference in regression \citep{candes2018panning,lei2021assumption}.

 \paragraph{Notation.} For any integer $n\geq 1$, $[n]$ denotes the set $\{1,\hdots,n\}$. For a vector $x\in\R^n$, 
 we define $\bar{x} = \frac{1}{n}\sum_{i=1}^n x_i$ as its mean and $\sigma^2_x = \frac{1}{n}\sum_{i=1}^n (x_i - \bar{x})^2$ as its variance.
 For any $i\in[n]$, we will also write $x_{\geq i}$ to denote the subvector $(x_i,\hdots,x_n)\in\R^{n-i+1}$ (with $x_{\geq 1} = x$),
 and $x_{\leq i}$ for the subvector $(x_1,\hdots,x_i)\in\R^i$ (with $x_{\leq n} = x$). 
 Finally, $\eqd$ denotes equality in distribution.

\section{Background}\label{sec:background}
In this section, we review the well-known Hoeffding and Bernstein concentration inequalities for the i.i.d.\ setting,
and some extensions to sampling without replacement.

\subsection{Hoeffding's inequality and Bernstein's inequality for the i.i.d.\ setting}
Consider data points  $X_i\iidsim P$ for some distribution $P$ on $[-1,1]$.
As for the exchangeable case, we will study the distribution of the weighted sum $\sum_{i=1}^n w_i X_i$,
for fixed weights $w_1,\hdots,w_n$.

First we review the classical Hoeffding bound on the weighted sum.
 \begin{theorem}[{\citet{hoeffding1963probability}}]\label{thm:Hoeffding_iid} 
 Let $w_1,\hdots,w_n$ be fixed and let $X_1,X_2,\hdots\iidsim P$, where $P$ is a distribution on $[-1,1]$
 with mean $\mu_P$.
For any $\lambda\in\R$,
\[\EE{\exp\left\{\lambda \sum_{i=1}^n w_i (X_i-\mu_P)\right\}}\leq \exp\left\{\frac{\lambda^2\|w\|^2_2}{2}\right\}.\]
Consequently, for any $\delta\in(0,1)$,
$\PP{\sum_{i=1}^n w_i (X_i-\mu_P) \geq  \|w\|_2 \sqrt{2\log(1/\delta)}} \leq \delta$.
 \end{theorem}

Next,  we review Bernstein's inequality, which
can give tighter bounds than the Hoeffding bound in
 settings where the distribution $P$ has low variance.
  \begin{theorem}[{\citet{bernstein1946theory}}]\label{thm:Bernstein_iid}
  Let $w_1,\hdots,w_n$ be fixed and let $X_1,X_2,\hdots\iidsim P$, where $P$ is a distribution on $[-1,1]$
 with mean $\mu_P$ and variance $\sigma^2_P$.
For any $\lambda\in\R$ with $|\lambda|< \frac{3}{2\|w\|_\infty}$,
\[\EE{\exp\left\{\lambda \sum_{i=1}^n w_i (X_i-\mu_P)\right\}}\leq \exp\left\{\frac{\lambda^2}{2(1-\frac{2\|w\|_\infty}{3}|\lambda|)}\cdot \sigma^2_P\|w\|^2_2\right\}.\]
 Consequently, for any $\delta\in(0,1)$,
\[\PP{\sum_{i=1}^n w_i (X_i-\mu_P) \geq \sigma_P\|w\|_2 \sqrt{2\log(1/\delta)} + \frac{2}{3}\|w\|_\infty  \log(1/\delta)}\leq \delta.\]
\end{theorem}

\subsection{Do these bounds extend to exchangeable data?}\label{sec:background_sampling_WOR}
Before turning to our main results, we first pause to ask whether these existing bounds
hold also in the broader setting of exchangeable, rather than i.i.d., data. 
In fact, this question includes a very important special case: the unweighted sum, $\sum_{i=1}^n X_i$,
or equivalently, weights $w_1 = \hdots = w_n=1$. 
This problem has a long history in the literature, and is typically studied in the context of sampling without replacement
from a finite population of size $N$. In fact, sampling without replacement has \emph{stronger} concentration than an i.i.d.\ 
sample---for instance, Serfling's inequality \citep{serfling1974probability} establishes that
\begin{equation}\label{eqn:Serfling_WOR}\EE{\exp\left\{\lambda \sum_{i=1}^n (X_i-\bar{X})\right\}}\leq \exp\left\{\frac{n\lambda^2}{2} \cdot \left(1-\frac{n-1}{N}\right)\right\},\end{equation}
which is strictly stronger than Theorem~\ref{thm:Hoeffding_iid}'s Hoeffding bound for the i.i.d.\ case.
More generally, for any exchangeable $X_1,\hdots,X_N$, the Hoeffding and Bernstein bounds for the i.i.d.\ case still hold for the unweighted sum $\sum_{i=1}^n X_i$ \cite[Theorem 4]{hoeffding1963probability} (again with concentration around $\bar{X}$). See also \cite{ramdas2023randomized} for additional concentration inequalities for sums of exchangeable random variables

However, in the setting of general weights, this is no longer the case:
as we will see in the following example, the bounds for the i.i.d.\ case
may no longer hold. 
\begin{example}\label{example:n=2}
Let $N=n=2$, and let $w_1=1$, $w_2 = -1$. If $X_1,X_2$ are sampled i.i.d.\ from the uniform distribution on $\{1,-1\}$,
then  by the Hoeffding bound (Theorem~\ref{thm:Hoeffding_iid}),
 \[\EE{e^{\lambda \sum_{i=1}^2 w_i X_i}} \leq e^{\frac{\lambda^2\|w\|^2_2}{2}}.\]
If instead $X_1,X_2$ are sampled uniformly without replacement from $\{1,-1\}$, then 
 \[\EE{e^{\lambda \sum_{i=1}^2 w_i X_i}} = \EE{e^{\lambda(X_1 - X_2)}} =  \frac{e^{2\lambda} + e^{-2\lambda}}{2} \approx e^{2\lambda^2}  = e^{2 \cdot \frac{\lambda^2\|w\|^2_2}{2}},\]
 where the approximation holds for $\lambda\approx 0$. 
\end{example}

To summarize, we have seen that the unweighted case and the weighted case are fundamentally quite different.
 In the unweighted setting,  the sum $\sum_{i=1}^n X_i$
 concentrates more strongly
 than in the i.i.d.\ case, and the main focus of the existing literature on this problem
 is to establish bounds for sampling without replacement that are \emph{tighter} than the i.i.d.\ bounds (such as Serfling's inequality~\eqref{eqn:Serfling_WOR};
 see also \cite[Corollary 3.1]{waudby2020confidence} for another example of a
 bound that is tighter than the i.i.d.\ bound).
  In contrast, in the weighted setting where the $w_i$'s are arbitrary,  Example~\ref{example:n=2}
demonstrates that if the $X_i$'s are exchangeable rather than i.i.d.\ then
 the sum $\sum_{i=1}^n w_i X_i$ may 
not satisfy the i.i.d.\ bounds.
Our goal is therefore to derive bounds that 
 are \emph{as close as possible} to the i.i.d.\ bounds.

 \section{Main results}\label{sec:main_results}
We now present our new bounds for the case of exchangeable data:
a Hoeffding bound in Section~\ref{sec:main_results_Hoeffding}, and a Bernstein
bound in Section~\ref{sec:main_results_Bernstein}. We will then see in Section~\ref{sec:main_results_exch_to_iid}
how these results lead to the classical i.i.d.\ bounds as a special case.

 \subsection{A Hoeffding bound for the exchangeable case}\label{sec:main_results_Hoeffding}
 Our first result is a Hoeffding type bound. First, for any $N\geq 2$ we define
 \[\eps_N = \frac{H_N-1}{N-H_N}\textnormal{ where $H_N = 1 + \frac{1}{2} + \hdots + \frac{1}{N}$}.\]
 Since the harmonic series scales as $H_N\asymp \log N$, this means that $\eps_N = \mathcal{O}(\frac{\log N}{N})$.
 
  \begin{theorem}\label{thm:Hoeffding_exch}
 Let $w\in\R^n$ be fixed and let $X_1,\hdots,X_N\in[-1,1]$ be exchangeable, where $N\geq n \geq 2$. Then
 for any $\lambda\in\R$,
 \[\EE{\exp\left\{ \lambda\sum_{i=1}^n w_i (X_i - \bar{X})\right\}}\leq\exp\left\{\frac{\lambda^2}{2}\|w\|^2_2\cdot (1+\eps_N)\right\}.\]
 Consequently, for any $\delta\in(0,1)$,
 $\PP{\sum_{i=1}^n w_i (X_i-\bar{X}) \geq  \|w\|_2  \sqrt{2(1+\eps_N)\cdot \log(1/\delta)}} \leq \delta$.
\end{theorem}

\subsubsection{A special case: nonnegative weights}
Recall from Section~\ref{sec:background_sampling_WOR} that the 
the i.i.d.\ Hoeffding bound (Theorem~\ref{thm:Hoeffding_iid}) holds for the unweighted case, $w_1=\hdots = w_n= 1$,
under the broader assumption of exchangeability.
While Example~\ref{example:n=2} demonstrates that this is no longer true 
 in general for arbitrary weights, we will now see that nonnegative weights avoid this issue.
   \begin{theorem}\label{thm:Hoeffding_exch_nonnegative}
 Let $w\in\R^n$ be fixed and let $X_1,\hdots,X_n\in[-1,1]$ be exchangeable. If $w_i\geq 0$ for all $i$, then
 for any $\lambda\in\R$, 
 \[\EE{\exp\left\{ \lambda\sum_{i=1}^n w_i (X_i - \bar{X})\right\}}\leq\exp\left\{\frac{\lambda^2}{2}\|w\|^2_2\right\}.\]
 Consequently, for any $\delta\in(0,1)$,
  $\PP{\sum_{i=1}^n w_i (X_i-\bar{X}) \geq  \|w\|_2  \sqrt{2\log(1/\delta)}} \leq \delta$.
\end{theorem}

 \subsection{A Bernstein bound for the exchangeable case}\label{sec:main_results_Bernstein}
 We next present a Bernstein type bound, which offers tighter bounds in settings where $\sigma^2_X$ is small.
   \begin{theorem}\label{thm:Bernstein_exch}
 Let $w\in\R^n$ be fixed and let $X_1,\hdots,X_N\in[-1,1]$ be exchangeable, where $N\geq n \geq 2$.  Then
 for any $\lambda\in\R$ with $|\lambda|< \frac{3}{2\|w\|_\infty(1+\eps_N)}$,
\[ \EE{\exp\left\{  \lambda\sum_{i=1}^n w_i (X_i - \bar{X})- \frac{\lambda^2(1+\eps_N)}{2(1-\frac{2|\lambda|}{3}\|w\|_\infty(1+\eps_N))}\cdot \tilde\sigma^2_{X,N} \|w\|^2_2\right\}}\leq 1,\]
where $\tilde{\sigma}_{X,N}^2 = \sigma^2_X+4\eps_N$. 
Consequently, for any $\delta\in(0,1)$,
\[\PP{\sum_{i=1}^n w_i (X_i-\bar{X}) \geq \tilde\sigma_{X,N} \|w\|_2 \sqrt{2(1+\eps_N) \log(1/\delta)} + \frac{2}{3}\|w\|_\infty (1+\eps_N) \log(1/\delta)} \leq \delta.\]
\end{theorem}

\subsection{From exchangeable to i.i.d.}\label{sec:main_results_exch_to_iid}

Next, we will show that the original i.i.d.\ versions of the Hoeffding and Bernstein bounds, given in Theorems~\ref{thm:Hoeffding_iid} 
and~\ref{thm:Bernstein_iid},
 can be recovered from our new results for exchangeable data, by taking $N\to \infty$.
 This is the goal of aim (b) (from Section~\ref{sec:intro}).
 
Consider an infinite sequence $X_1,X_2,\hdots\iidsim P$.
Recalling the notation $X_{\leq N} = (X_1,\dots,X_N)$,
for the Hoeffding bound we then have
\begin{multline*}
\EE{\exp\left\{ \lambda\sum_{i=1}^n w_i (X_i - \mu_P) - \frac{\lambda^2}{2}\|w\|^2_2\right\}}\\ 
=\lim_{N\to\infty} \EE{ \exp\left\{ \lambda\sum_{i=1}^n w_i (X_i - \bar{X}_{\leq N}) -\frac{\lambda^2}{2}\|w\|^2_2\cdot(1+\eps_N) \right\}} \leq 1,\end{multline*}
where for the first step we apply the dominated convergence theorem along with the fact that
 $\bar{X}_{\leq N}\to\mu_P$ almost surely by the strong law of large numbers (and also $\eps_N\to 0$), while  the last step applies our new result,
 Theorem~\ref{thm:Hoeffding_exch}, to the exchangeable finite sequence $X_{\leq N}$. This is equivalent to the result of Theorem~\ref{thm:Hoeffding_iid} for the i.i.d.\ case.

Next we turn to the Bernstein bound. Assume $|\lambda|< \frac{2}{3\|w\|_\infty}$ (and so, since $\eps_N\to 0$,
for sufficiently large $N$ we have $|\lambda|< \frac{2}{3\|w\|_\infty(1+\eps_N)}$). Then
\begin{multline*}
\EE{\exp\left\{ \lambda\sum_{i=1}^n w_i (X_i - \mu_P) - \frac{\lambda^2}{2(1-\frac{2|\lambda|}{3}\|w\|_\infty)}\sigma^2_P\|w\|^2_2\right\}}\\ 
= \lim_{N\to\infty}\EE{ \exp\left\{  \lambda\sum_{i=1}^n w_i (X_i - \bar{X}_{\leq N})- \frac{\lambda^2(1+\eps_N)}{2(1-\frac{2|\lambda|}{3}\|w\|_\infty(1+\eps_N))} \tilde\sigma^2_{X_{\leq N},N}\|w\|^2_2\right\}}
\leq 1,\end{multline*}
by similar arguments as for the Hoeffding bound (since $\bar{X}_{\leq N}\to\mu_P$ and  $\tilde\sigma^2_{X_{\leq N},N}\to\sigma^2_P$ almost surely).
This recovers Theorem~\ref{thm:Bernstein_iid} for the i.i.d.\ case.

\section{Related work}\label{sec:related_work}

\subsection{Existing results for special cases}\label{sec:related_work_special_cases}
We first summarize some existing results for special cases of the problem.

\paragraph{The i.i.d.\ setting.} The i.i.d.\ setting, where $X_1,X_2,\hdots\iidsim P$ for some distribution $P$,
is of course a special case of exchangeability.
In addition to the classical Hoeffding bound \citep{hoeffding1963probability} and Bernstein bound \citep{bernstein1946theory} for the i.i.d.\ setting (presented 
in Theorems~\ref{thm:Hoeffding_iid} and~\ref{thm:Bernstein_iid}),
many other types of concentration bounds are well established in the literature, and may provide sharper bounds
for certain settings. See \citet{boucheron2013concentration} for background on a range of such bounds.
More recent literature has considered the extension to time-uniform bounds---providing bounds that hold
uniformly over a sequence of times $n$ rather than at a single fixed $n$.
In particular, the work of \citet{howard2020time,howard2021time} provides a martingale based approach
to this problem (this is closely related to time-uniform bounds for sampling without replacement, studied
by \citet{waudby2020confidence}, whose work we describe in more detail below).

\paragraph{The unweighted sum setting.}
Another special case is that of an unweighted sum, $\sum_{i=1}^n X_i$.
As we discussed earlier in Section~\ref{sec:background_sampling_WOR} (see~\eqref{eqn:Serfling_WOR}), 
in the setting of sampling without replacement,
 Hoeffding's inequality can be strengthened further to Serfling's inequality \citep{serfling1974probability}.
Multiple works in the more recent literature have obtained additional refined bounds for this problem.
For example, the work of \citet{bardenet2015concentration} gives a tighter version of the Hoeffding--Serfling inequality,
and also derives some Bernstein type bounds.
Stronger Bernstein type results are also established by \citet{greene2017exponential}, 
via relating the distribution of the unweighted sum $\sum_{i=1}^n X_i$ to a Hypergeometric distribution.

While these results consider a problem that is clearly closely related to the central question of this work,
the general case (where weights $w_i$ may not be constant) is a more challenging problem---as we saw in Example~\ref{example:n=2},
the unweighted case is fundamentally different from the general case, where the $w_i$'s may vary.

\subsection{Existing results for weighted sums}\label{sec:related_work_weighted_sums}
Next, we compare to existing results for the more general problem considered in this paper---the weighted
sum $\sum_{i=1}^n w_i X_i$, for arbitrary weights $w_i$ (and exchangeable $X_i$'s).

First, we present the result of  \cite[Theorem 1]{gan2010fast}, which proves
a Hoeffding type bound using an elegant argument based on Stein's exchangeable pairs.
Translating their result into our notation, their theorem implies that
\[\PP{ \left| \sum_{i=1}^n w_i (X_i - \bar{X}) \right| \geq \|w\|_2 \sqrt{4\log(2/\delta)} } \leq \delta,\]
for any $\delta\in(0,1)$. 
Compare this to our tail bound in Theorem~\ref{thm:Hoeffding_exch}, where we have a factor $2(1+\eps_N)$ in place of the factor of $4$.
In particular, since $\eps_N\to 0$ (and $\eps_N\leq 1$ for all $N\geq 2$), our bound is strictly tighter, and is novel in being able to achieve
the same constant as the 
i.i.d.\ bound in the limit $N\to\infty$ (as in aim (b)).

We next turn to the recent results of 
\citet{waudby2020confidence}. The central problem studied in this work is different from our work.
If we observe $X_1,\hdots,X_n$ as a sample drawn without replacement from a finite population of size $N$,
our aim is to bound deviations of $\sum_{i=1}^n w_i X_i$ for \emph{fixed} weights $w_i$. In particular, for nonconstant
weights, this question is nontrivial even if $\bar{X}$, the mean of the finite population,  is known.
In contrast, \citet{waudby2020confidence}'s aim is to estimate an \emph{unknown} mean $\bar{X}$ using the sample $X_1,\hdots,X_n$.
They do this by  \emph{designing} weights $w_1,\hdots,w_n$ such that the sum, $\sum_{i=1}^n w_i X_i$,
can be used to construct a tight confidence intervals for $\bar{X}$.
Despite these differences, these two questions are closely related, since a bound on
$\left|\sum_{i=1}^n w_i (X_i - \bar{X})\right|$
(as we provide in our main results here) can provide as a confidence interval for the mean $\bar{X}$.

Another key distinction 
is that \citet{waudby2020confidence}
study the streaming setting, providing a time-uniform guarantee (i.e.,
a valid confidence interval for $\bar{X}$ simultaneously at every $n=1,\hdots,N$), 
which is a stronger guarantee than the fixed-$n$ case considered
in this work. As a consequence, though, the problem they study is not invariant to permuting indices,
due to the sequential nature of the streaming setting.
In contrast, as described in aim (c)  in Section~\ref{sec:intro}, since we are only considering a fixed value of $n$ in our work,
any interpretable bound in our setting should be invariant to reordering the indices.
However,
the results of \citet{waudby2020confidence} provide some key technical tools that we will use in this work 
for proving our main results.
We remark also that \cite[Appendix D]{shekhar2023risk} extend these results to allow for a setting
where the $X_i$'s are sampled with a weighted, rather than uniform, distribution.

\paragraph{The matrix setting.}
A more challenging problem that has also been considered in the literature
is the following: given a fixed matrix $A\in\R^{n\times n}$, derive concentration bounds
for the sum
$\sum_{i=1}^n A_{i,\pi(i)}$,
where the permutation $\pi\in\mathcal{S}_n$ is sampled uniformly at random (here $\mathcal{S}_n$ denotes
the set of permutations of $[n]$).
To connect this to the problem of bounding a weighted sum $\sum_{i=1}^n w_i X_i$,
we consider a \emph{fixed} vector $x=(x_1,\hdots,x_n)\in[-1,1]^n$, and let $X_i = x_{\pi(i)}$ be defined by a
random permutation of $x$ (equivalently, by sampling without replacement from the entries of $x$).
If we then consider the rank-$1$ matrix $A=wx^\top$, we have $\sum_{i=1}^n w_i X_i = \sum_{i=1}^n A_{i,\pi(i)}$.
In other words, any result bounding the sum in the matrix setting, can be applied to the setting
of a weighted sum of exchangeable random variables as a special case.

Examples of results of this type include those of \citet{chatterjee2007stein} and \citet{albert2019concentration}
for the case $A_{ij}\in[0,1]^{n\times n}$, which are tight when the values of $A$ mostly concentrate near zero---this setting is
therefore not directly
comparable to our work.
The more recent work of  \citet{polaczyk2023concentration} offers a result that is more related to the problem studied here: they consider a matrix $A\in[-1,1]^{n\times n}$,  and assume that the matrix is centered, $\sum_{i=1}^n\sum_{j=1}^n A_{ij}=0$.
They prove
\[\PP{ \sum_{i=1}^n A_{i,\pi(i)} \geq t } \leq  2 \exp\left\{ - \frac{t}{36}\cdot \log\left(1 + \frac{t}{\frac{36}{n}\sum_{i=1}^n \sum_{j=1}^n A_{ij}^2}\right)\right\}.\]
To compare this to our work, suppose $\|w\|_\infty\leq 1$, and assume again that $X_i = x_{\pi(i)}$ for a fixed vector $x_1,\hdots,x_n$.
Consider the simple case $\bar{x}=0$. 
Then $A= wx^\top$ is centered and takes values in $[-1,1]$. Applying  \citet{polaczyk2023concentration}'s bound,
and using the fact that $\log(1+\eps)\geq\frac{\eps}{1+\eps}$ for any $\eps>0$
to simplify the expression, we obtain
\[\PP{ \sum_{i=1}^n w_i (X_i-\bar{X}) \geq 36\sigma_X\|w\|_2\sqrt{\log(2/\delta)} + 36\log(2/\delta) } \leq \delta\]
for any $\delta\in(0,1)$. Up to constants,
this  has the same form as our Bernstein tail bound for the exchangeable setting (Theorem~\ref{thm:Bernstein_exch}).
However, the constants here are substantially larger, and in particular will not recover
the classical bound for the i.i.d.\ case.

\subsection{Additional related work}\label{sec:related_work_additional}
Finally, we mention a related line of work, which studies the problem of establishing a central limit theorem 
for sums of exchangeable, rather than i.i.d.\ (or independent), random variables.
\citet{hajek1960limiting} establishes a central limit theorem for the unweighted case,
$\sum_{i=1}^n X_i$ (where $X_i$'s are sampled without replacement from a finite population). 
\citet{bolthausen1984estimate} studies the problem for the matrix case, establishing a central
limit theorem for the sum $\sum_{i=1}^n A_{i,\pi(i)}$.
More recent work by  \citet{li2017general} considers an extension to a problem in causal inference, for
the setting of random treatment assignment into multiple groups.

 \section{Discussion}\label{sec:discussion}
In this work, we have established new versions of Hoeffding and Bernstein inequalities for weighted
 sums of exchangeable random variables, $w_1 X_1 + \hdots + w_n X_n$ for an exchangeable
 sequence $X_1, \hdots , X_n,X_{n+1},\hdots,X_N$. 
 Compared to prior work, these new bounds offer the benefit of reducing to existing results for the i.i.d.\ setting
in the limit $N\to \infty$. Essentially, we can think of these results as providing a continuous interpolation between
the setting of a random permutation (where $N=n$, and so $X_1,\hdots,X_n$ is simply a random permutation of a finite list)
and the setting of i.i.d.\ sampling (where $N=\infty$).

Examining these results in the context of the existing literature raises a number of open questions
and potential extensions. 
First, it is natural to ask whether the bounds are tight: for instance, are the inflation terms $(1+\eps_N)$ unavoidable?
We have seen that in the special case of nonnegative weights $w_i\geq 0$, this inflation can be avoided for the Hoeffding
bound (and the i.i.d.\ Hoeffding bound is recovered). However, it is unclear if the same might be true for the Bernstein bound
with nonnegative weights. 

Next, for the i.i.d.\ setting, recent lines of work have established concentration 
bounds that are stronger or more flexible than the Hoeffding and Bernstein bounds---for example, empirical Bernstein
inequalities, where the variance $\sigma^2_P$ of the distribution can be replaced by an empirical estimate
(see \citet{boucheron2013concentration} for a survey of these types of results).
\citet{waudby2020confidence}'s recent work provides empirical Bernstein bounds for exchangeable data in the streaming setting.
 Can empirical Bernstein type bounds be derived for exchangeable data in a way that satisfies
 the aims (a), (b), and (c) introduced in Section~\ref{sec:intro}?
 
 Furthermore, as mentioned in Section~\ref{sec:intro}, finite versions of De Finetti's theorem
 \citep{diaconis1980finite,gavalakis2021information} bound the distance between the distribution of 
 $X_1,\dots,X_n$ (a subsequence of $N\gg n$ exchangeable random variables), and a mixture of i.i.d.\ distributions.
It is interesting to examine how these types of results connect to the questions of this paper, and 
whether these results may be used to derive additional concentration properties of the weighted sum $w_1X_1 + \dots + w_nX_n$
studied in this work.

Finally, in Section~\ref{sec:related_work} we also presented some existing results 
on a matrix version of the problem, where the sum is given by $\sum_{i=1}^n A_{i,\pi(i)}$ for a fixed matrix $A\in\R^{n\times n}$
and a random permutation $\pi$. It is possible that the technical tools developed
in this work could be extended to derive new bounds for this matrix setting as well.

\appendix

\section{Proof of the Hoeffding bound}\label{sec:app_proofs1}
In this section we prove Theorem~\ref{thm:Hoeffding_exch}, the Hoeffding bound for the exchangeable case.
We first review a martingale-based bound on the moment-generating function (MGF) that was established by \citet{waudby2020confidence}.

\begin{proposition}[{\cite[Theorem 3.1]{waudby2020confidence}}]\label{prop:Hoeffding_martingale}
Fix $N\geq n\geq 1$ and  $v\in\R^n$, and
let $X=(X_1,\hdots,X_n)$ where  $X_1,\hdots,X_n\in[-1,1]$ are exchangeable. 
Fix any $\lambda\in\R$, and define $M_0=1$ and
\[M_k = \exp\left\{ \lambda \sum_{i=1}^k v_i (X_i - \bar{X}_{\geq i}) - \frac{\lambda^2}{2}\sum_{i=1}^k v_i^2\right\}\]
for $k\in[n]$.
Then $M_k$ is a supermartingale, i.e., $\EEst{M_k}{M_{k-1}} \leq M_{k-1}$ for all $k\in[n]$.
\end{proposition}

As we will see shortly, this bound cannot be immediately applied to establish the type 
of results that are the aim of this paper, but it provides an important step.

\begin{proof}[Proof of Theorem~\ref{thm:Hoeffding_exch}]
First we verify that
it suffices to consider the case $N=n$.
To see why, suppose $n<N$. 
Define 
$\tilde{w} = (w_1,\hdots,w_n,0,\hdots,0)\in\R^N$.
Since $\|\tilde{w}\|_2=\|w\|_2$, and $\sum_{i=1}^N \tilde{w}_i (X_i - \bar{X}) = \sum_{i=1}^n w_i(X_i - \bar{X})$,
 this means
that the result of Theorem~\ref{thm:Hoeffding_exch}
(applied with weight vector $w\in\R^n$) is equivalent to the results of the theorem when applied with weight vector $\tilde{w}\in\R^N$.
Therefore, from this point on, we only  consider the case $N=n$.

Define a matrix $B_n\in\R^{n\times n}$, with $i$th row $(\mathbf{0}_{i-1},  \frac{1}{n+1-i}\cdot \mathbf{1}_{n+1-i})$, where $\mathbf{1}_k$ and $\mathbf{0}_k$ denote the vector of 1's and the vector of 0's, respectively, of length $k$.
This matrix provides the following linear operator: for each $i\in[n]$,
$(B_n X)_i = \bar{X}_{\geq i}$, where  $X = (X_1,\hdots,X_n)$. Thus for any $v\in\R^n$,
$ \sum_{i=1}^n v_i (X_i - \bar{X}_{\geq i}) 
= v^\top (\mathbf{I}_n - B_n) X$.
Define also $A_n\in\R^{n\times n}$ with $i$th row $(\mathbf{0}_{i-1},1,-\frac{1}{n-i}\cdot\mathbf{1}_{n-i})$ for $i\in[n-1]$, and $n$th row $\mathbf{0}_n$.
We can verify that $A_n^\top (\mathbf{I}_n - B_n) =  \mathcal{P}^\perp_{\mathbf{1}_n}$,
where $ \mathcal{P}^\perp_{\mathbf{1}_n}$
denotes projection to the subspace orthogonal to $\mathbf{1}_n$.

Now fix any $u\in\R^n$ and let $v = A_n u$. We then have
\[\sum_{i=1}^n u_i (X_i - \bar{X}) = u^\top \mathcal{P}^\perp_{\mathbf{1}_n} X
= u^\top A_n^\top (\mathbf{I}_n - B_n) X = v^\top (\mathbf{I}_n - B_n) X =  \sum_{i=1}^n v_i (X_i - \bar{X}_{\geq i})  .\]
Therefore,
\begin{equation}\label{eqn:Hoeffding_bound_with_An} \EE{\exp\left\{ \lambda \sum_{i=1}^n u_i (X_i - \bar{X}) -\frac{\lambda^2}{2}\|A_n u\|^2_2\right\}}
= \EE{\exp\left\{\lambda \sum_{i=1}^n v_i(X_i-\bar{X}_{\geq i}) - \frac{\lambda^2}{2}\|v\|^2_2\right\}}\leq1,\end{equation}
where the last step holds by \cite[Theorem 3.1]{waudby2020confidence} (restated in Proposition~\ref{prop:Hoeffding_martingale} above---we are using the fact that, in the notation of the proposition, $\EE{M_n}\leq \EE{M_0}=1$).
We pause here to remark that this inequality could be used to bound the MGF of our weighted sum $\sum_{i=1}^n w_i (X_i - \bar{X}_n)$
(by taking $u=w$), but 
this form of the result would not satisfy aim (c), since $\|A_n w\|^2_2$ is not invariant to permuting the weights $w_1,\hdots,w_n$.

Now we will apply~\eqref{eqn:Hoeffding_bound_with_An} with $u = (A_n^\top A_n)^{\dagger} \Pi w$, where $\Pi\in\{0,1\}^{n\times n}$ is any permutation matrix, while $\dagger$ denotes the matrix pseudoinverse. 
Note that, since $A_n \mathbf{1}_n = 0$ by construction, this means that $ (A_n^\top A_n)^{\dagger} \mathbf{1}_n=0$ as well---this implies $u^\top\mathbf{1}_n =0$ and so $\sum_{i=1}^n u_i (X_i - \bar{X}) = u^\top X = w^\top \Pi^\top (A_n^\top A_n)^{\dagger} X$.
We then have
\[\EE{\exp\left\{ \lambda w^\top \Pi^\top (A_n^\top A_n)^{\dagger} X-\frac{\lambda^2}{2}\|A_n (A_n^\top A_n)^{\dagger} \Pi w\|^2_2\right\}}\leq 1,\]
 for each fixed permutation matrix $\Pi$.
Since $X\eqd \Pi X$ by the exchangeability assumption, this also holds with $\Pi X$ in place of $X$, for each $\Pi$. Taking the average over all $n!$ possible permutation matrices $\Pi$, and applying Jensen's inequality, we have
\[\EE{\exp\left\{  \frac{1}{n!}\sum_{\Pi} \left(\lambda w^\top \Pi^\top (A_n^\top A_n)^{\dagger}  \Pi X-\frac{\lambda^2}{2} w^\top \Pi^\top  (A_n^\top A_n)^{\dagger} \Pi w\right)\right\}}\leq 1.\]
An elementary calculation shows
$ \frac{1}{n!}\sum_{\Pi}   \Pi^\top (A_n^\top A_n)^{\dagger}\Pi = \frac{n-H_n}{n-1} \cdot \mathcal{P}^\perp_{\mathbf{1}_n} = \frac{1}{1+\eps_n} \cdot \mathcal{P}^\perp_{\mathbf{1}_n}$, and so
\begin{multline*}1\geq \EE{\exp\left\{\frac{1}{1+\eps_n} \cdot  \left(\lambda w^\top \mathcal{P}^\perp_{\mathbf{1}_n} X-\frac{\lambda^2}{2} w^\top  \mathcal{P}^\perp_{\mathbf{1}_n} w\right)\right\}} \\
\geq  \EE{\exp\left\{ \frac{1}{1+\eps_n} \cdot  \left(\lambda \sum_{i=1}^n w_i (X_i - \bar{X})-\frac{\lambda^2}{2}\|w\|^2_2\right)\right\}} ,\end{multline*}
where the last step holds since $w^\top \mathcal{P}^\perp_{\mathbf{1}_n} X =  \sum_{i=1}^n w_i (X_i - \bar{X})$ and $w^\top  \mathcal{P}^\perp_{\mathbf{1}_n} w\leq \|w\|^2_2$.
Replacing $\lambda$ with $\lambda (1+\eps_n)$ and rearranging terms yields  the desired MGF bound.
Finally, the tail probability bound holds due to the following standard fact \cite[Section 2.3]{boucheron2013concentration}:
for any random variable $Z\in\R$,
\begin{equation}\label{eqn:Hoeffding_from_MGF_to_tail}
\textnormal{If $\EE{e^{\lambda Z}}\leq e^{\lambda^2 a^2/2}$ for all $\lambda\in\R$ then 
$\PP{Z\geq a\sqrt{2\log(1/\delta)}}\leq \delta$ for all $\delta\in(0,1)$.}\end{equation}
\end{proof}

\begin{proof}[Proof of Theorem~\ref{thm:Hoeffding_exch_nonnegative}]
We will use the bound~\eqref{eqn:Hoeffding_bound_with_An}, which was derived as a preliminary step in the proof of Theorem~\ref{thm:Hoeffding_exch}.
Without loss of generality, assume that we have $w_1 \geq \hdots \geq w_n \geq 0$.
Applying~\eqref{eqn:Hoeffding_bound_with_An} with $u=w$, we have
$\EE{\exp\left\{ \lambda \sum_{i=1}^n w_i (X_i - \bar{X})\right\}} \leq \exp\left\{\frac{\lambda^2}{2}\|A_n w\|^2_2\right\}$.
For each $i\in[n]$, by definition of $A_n$ we have 
$\big((A_n w)_i\big)^2 = (w_i - \bar{w}_{\geq i})^2 \leq w_i^2$,
where the last step holds since, by our assumption on $w$, we have
$w_i \geq \bar{w}_{\geq i} \geq 0$.
Therefore, $\|A_n w\|^2_2 \leq \|w\|^2_2$, which completes the proof 
of the MGF bound. Finally, the tail probability bound follows immediately by applying the fact~\eqref{eqn:Hoeffding_from_MGF_to_tail}.
\end{proof}

 \section{Proof of the Bernstein bound}\label{sec:app_proofs2}

Recall that, for our proof of the Hoeffding bound, a key ingredient
was the martingale based bound of {\cite[Theorem 3.1]{waudby2020confidence}} (restated
in Proposition~\ref{prop:Hoeffding_martingale}). 
For the Bernstein MGF proof, we will need to follow a similar route. We begin by presenting a Bernstein type version of the
martingale result.
This result builds on recent techniques and related bounds from the literature, e.g., \cite[Theorem 3.2]{waudby2020confidence},
\cite[Corollary 1(c)]{howard2020time}, \cite[Section 5]{waudby2024estimating}.

\begin{proposition}\label{prop:Bernstein_martingale}
Fix $n\geq 1$ and  $v\in\R^n$, and let $X=(X_1,\hdots,X_n)$ where  $X_1,\hdots,X_n\in[-1,1]$ are exchangeable. 
Fix any $\lambda\in\R$ satisfying $|\lambda|<\frac{2}{3\|v\|_\infty}$, and define $M_0=1$, and
\[M_k = \exp\left\{ \lambda \sum_{i=1}^k v_i (X_i - \bar{X}_{\geq i})-\frac{\lambda^2}{2(1-\frac{2|\lambda|}{3}\|v\|_\infty)}\sum_{i=1}^k v_i^2\sigma^2_{X_{\geq i}}\right\}\textnormal{ for $k\in[n]$}.\]
Then $M_k$ is a supermartingale, i.e., $\EEst{M_k}{M_{k-1}} \leq M_{k-1}$ for all $k\in[n]$.
\end{proposition}

\begin{proof}[Proof of Proposition~\ref{prop:Bernstein_martingale}]
By construction, for each $k\in[n]$ we can write
\[M_k = M_{k-1} \cdot \exp\left\{ \lambda v_k (X_k - \bar{X}_{\geq k})  - \frac{\lambda^2}{2(1-\frac{2|\lambda|}{3}\|v\|_\infty)} v_k^2\sigma^2_{X_{\geq k}}\right\}.\]
Define a filtration $\{\Fcal_k\}_{k\in\{0,\hdots,n\}}$, where $\Fcal_k$ is the $\sigma$-algebra
generated by $\bar{X},\sigma^2_X,X_1,\hdots,X_k$. 
For any $i\leq k$, 
$\bar{X}_{\geq i} = \frac{X_i + \hdots + X_n}{n-i+1} = \frac{n\bar{X} - (X_1 + \hdots + X_{i-1})}{n-i+1}$,
 so
$\bar{X}_{\geq i}$ is $\Fcal_{k-1}$-measurable.
A similar calculation verifies that $\sigma^2_{X_{\geq i}}$ is $\Fcal_{k-1}$-measurable, as well.
This then implies $M_k$ is $\Fcal_k$-measurable
for each $k$. Conditional on $\Fcal_{k-1}$, the random variables $X_k ,\hdots,X_n$ are exchangeable,
and so $X_k$ has conditional mean $\bar{X}_{\geq k}$ and conditional variance $\sigma^2_{X_{\geq k}}$.
And, since the $X_i$'s take values in $[-1,1]$,  $|v_k(X_k - \bar{X}_{\geq k})|\leq 2\|v\|_\infty$ almost surely.
Then
\[M_{k-1}^{-1} \EEst{M_k}{\Fcal_{k-1}} =  \EEst{ \exp\left\{ \lambda v_k (X_k - \bar{X}_{\geq k}) - \frac{\lambda^2}{2(1-\frac{2|\lambda|}{3}\|v\|_\infty)} v_k^2\sigma^2_{X_{\geq k}}\right\}}{\Fcal_{k-1}}\leq  1,\]
by
 the basic Bernstein inequality, which states that if $\EE{Z}=0$ and $\textnormal{Var}(Z)=\sigma^2$ and $|Z|\leq B$ almost surely, then
$\EE{e^{t Z}} \leq e^{\frac{t^2\sigma^2}{2(1-|t|B/3)}}$
 for any $|t|<3/B$.\end{proof}

We are now ready to prove the theorem.
\begin{proof}[Proof of Theorem~\ref{thm:Bernstein_exch}]
It suffices to consider the case $N=n$, by an analogous argument as for the Hoeffding bound in Theorem~\ref{thm:Hoeffding_exch}.
First, we apply the result of Proposition~\ref{prop:Bernstein_martingale}:
for any $v\in\R^n$ and $|\lambda|<\frac{3}{2\|v\|_\infty}$, we have $\EE{M_n}\leq \EE{M_0}=1$,
which we can rewrite as
\[ \EE{\exp\left\{ \lambda v^\top(\mathbf{I}_n - B_n)X-\frac{\lambda^2}{2(1-\frac{2|\lambda|}{3}\|v\|_\infty)}\sum_{i=1}^nv_i^2\sigma^2_{X_{\geq i}}\right\}}\leq1,\]
where $B_n$ is defined as before.
Moreover, for any $i$, we have
\[\sigma^2_{X_{\geq i}} = \frac{1}{n-i+1}\sum_{j=i}^n (X_j - \bar{X}_{\geq i})^2 
\leq \frac{1}{n-i+1}\sum_{j=i}^n (X_j - \bar{X})^2 = \bar{\tilde{X}}_{\geq i} = (B_n\tilde{X})_i,\]
where we define a vector $\tilde X$ with entries $\tilde X_i = (X_i - \bar{X})^2$. Then
$ \sum_{i=1}^n v_i^2\sigma^2_{X_{\geq i}} \leq  \sum_{i=1}^n v_i^2 (B_n \tilde{X})_i = \tilde{v}^\top B_n \tilde{X}$,
where $\tilde{v}$ is the vector with entries $\tilde{v}_i = v_i^2$. This yields
 \[\EE{\exp\left\{ \lambda v^\top (\mathbf{I}_n - B_n)X - \frac{\lambda^2}{2(1-\frac{2|\lambda|}{3}\|v\|_\infty)} \tilde{v}^\top B_n \tilde{X}\right\}}\leq 1.\]

We now apply this bound with $v = \Pi w$, where $\Pi\in\{0,1\}^{n\times n}$ is any permutation matrix. Note that $\tilde{v} = \Pi \tilde{w}$
and $\|v\|_\infty = \|\Pi w\|_\infty = \|w\|_\infty$. We then have, for any $|\lambda|<\frac{3}{2\|w\|_\infty}$,
 \[\EE{\exp\left\{ \lambda w^\top \Pi^\top (\mathbf{I}_n - B_n) X  - \frac{\lambda^2}{2(1-\frac{2|\lambda|}{3}\|w\|_\infty)} \tilde{w}^\top \Pi^\top B_n \tilde{X}\right\}}\leq 1.\]
Since $X\eqd \Pi X$, we can replace $X$ with $\Pi X$ in the expression above to obtain
 \[\EE{\exp\left\{ \lambda w^\top \Pi^\top (\mathbf{I}_n - B_n) \Pi X  - \frac{\lambda^2}{2(1-\frac{2|\lambda|}{3}\|w\|_\infty)} \tilde{w}^\top \Pi^\top B_n \Pi \tilde{X}\right\}}\leq 1,\]
 since $\widetilde{\Pi X} = \Pi\tilde{X}$ by construction.
Taking an average over all $n!$ possible permutation matrices $\Pi$ and applying Jensen's inequality, we therefore have
\[ \EE{\exp\left\{  \frac{1}{n!}\sum_{\Pi} \left(\lambda w^\top \Pi^\top (\mathbf{I}_n - B_n) \Pi X  - \frac{\lambda^2}{2(1-\frac{2|\lambda|}{3}\|w\|_\infty)} \tilde{w}^\top \Pi^\top B_n \Pi \tilde{X}\right)\right\}}\leq 1.\]
A straightforward calculation shows
$\frac{1}{n!}\sum_\Pi \Pi^\top B_n \Pi  =\mathbf{I}_n -  \frac{1}{1+\eps_n} \mathcal{P}^\perp_{\mathbf{1}_n} =  \frac{1}{1+\eps_n}  \mathcal{P}_{\mathbf{1}_n}+ \frac{\eps_n}{1+\eps_n} \mathbf{I}_n$.
We can also calculate $ w^\top \mathcal{P}^\perp_{\mathbf{1}_n} X= \sum_{i=1}^n w_i (X_i - \bar{X})$, and also,
\[ \tilde{w}^\top \mathcal{P}_{\mathbf{1}_n} \tilde{X} = \sum_{i=1}^n \tilde{w}_i\cdot\frac{ \sum_{i=1}^n \tilde{X}_i}{n} 
 = \|w\|^2_2 \sigma^2_X,\quad \tilde{w}^\top \mathbf{I}_n \tilde{X} = \sum_{i=1}^n \tilde{w}_i\tilde{X}_i = \sum_{i=1}^n w_i^2 (X_i - \bar{X})^2 \leq 4\|w\|^2_2.\]
Plugging in these calculations, then,
\[ \EE{\exp\left\{  \frac{\lambda}{1+\eps_n}\sum_{i=1}^n w_i (X_i - \bar{X})- \frac{\lambda^2}{2(1-\frac{2|\lambda|}{3}\|w\|_\infty)} \left( \frac{ \|w\|^2_2 \sigma^2_X}{1+\eps_n} + \frac{\eps_n \cdot  4\|w\|^2_2}{1+\eps_n} \right)\right\}}\leq 1.\]
Replacing $\lambda$ with $\lambda(1+\eps_n)$ (and assuming now $|\lambda|<\frac{3}{2\|w\|_\infty(1+\eps_n)}$) completes the proof 
 of the MGF bound.

Finally, we need to verify the tail bound.
First, observe that $X=(X_1,\dots,X_N)$
is exchangeable conditional on $\tilde{\sigma}^2_{X,N}$ (because $\tilde{\sigma}^2_{X,N}$ is a symmetric function of $X$). 
Applying the MGF bound derived above to this \emph{conditional} distribution yields
\[ \EEst{\exp\left\{  \lambda\sum_{i=1}^n w_i (X_i - \bar{X})\right\}}{\tilde{\sigma}^2_{X,N}} \leq \exp\left\{ \frac{\lambda^2(1+\eps_N)}{2(1-\frac{2|\lambda|}{3}\|w\|_\infty(1+\eps_N))}\cdot \tilde\sigma^2_{X,N} \|w\|^2_2\right\}\]
for any $|\lambda|< \frac{3}{2\|w\|_\infty(1+\eps_N)}$.
We will next need apply the following standard fact
 \cite[Section 2.4]{boucheron2013concentration}:
for any random variable $Z\in\R$ and $a,b>0$,
\begin{multline}\label{eqn:Bernstein_from_MGF_to_tail}
\textnormal{If $\EE{e^{\lambda Z}}\leq e^{\frac{\lambda^2 a^2}{2(1-b|\lambda|)}}$ for all $|\lambda| < 1/b$,}\\\textnormal{then 
$\PP{Z\geq a\sqrt{2\log(1/\delta)}+b\log(1/\delta)}\leq \delta$ for all $\delta\in(0,1)$.}\end{multline}
Applying~\eqref{eqn:Bernstein_from_MGF_to_tail} with $a=\tilde\sigma_{X,N} \|w\|_2\sqrt{1+\eps_N}$ and $b=\frac{2}{3}\|w\|_\infty(1+\eps_N)$, then,
\[ \PPst{\sum_{i=1}^n w_i (X_i - \bar{X}) \geq \tilde\sigma_{X,N} \|w\|_2 \sqrt{2(1+\eps_N) \log(1/\delta)} + \frac{2}{3}\|w\|_\infty (1+\eps_N) \log(1/\delta) }{\tilde{\sigma}^2_{X,N}} \leq \delta.\]
After marginalizing over $\tilde\sigma^2_{X,N}$, this yields the desired bound.
\end{proof}

 \subsection*{Acknowledgements}
The author was
supported by the Office of Naval Research via grant N00014-20-1-2337 and by the National Science
Foundation via grant DMS-2023109.
 The author thanks Aaditya Ramdas and Ian Waudby-Smith for helpful discussions. 
 
\bibliographystyle{plainnat}
\bibliography{bib}
\end{document}